\documentclass{article}

\usepackage{amsmath}
\usepackage{amsfonts}
\usepackage{amsthm}
\usepackage{amssymb}
\usepackage{bm}
\usepackage[applemac]{inputenc}
\usepackage{enumerate}
\usepackage{a4wide}

\usepackage{mathpazo}

\newtheoremstyle{mystyle}{}{}{\slshape}{2pt}{\scshape}{.}{ }{} 

\newtheorem{thm}{Theorem}[section]
\newtheorem{corollary}[thm]{Corollary}
\newtheorem{cor}[thm]{Corollary}

\newtheorem{prop}[thm]{Proposition}
\newtheorem{lemme}[thm]{Lemma}
\newtheorem{lemma}[thm]{Lemma}

\newtheorem{fact}[thm]{Fact}

\newtheorem{conjecture}[thm]{Conjecture}

\theoremstyle{definition}
\newtheorem{defi}[thm]{Definition}

\theoremstyle{mystyle}

\theoremstyle{remark}

\newcommand{\impl}{\rightarrow}

\newcommand{\ignore}[1]{}

\DeclareMathOperator{\tp}{tp}

\DeclareMathOperator{\dcl}{dcl}
\DeclareMathOperator{\acl}{acl}

\DeclareMathOperator{\aut}{Aut}

\newcommand{\tleq}{\unlhd}

\def\indsym#1#2{%
 \setbox0=\hbox{$\m@th#1x$}%
 \kern\wd0%
 \hbox to 0pt{\hss$\m@th#1\mid$\hbox to 0pt{$\m@th#1^{#2}$\hss}\hss}%
 \lower.9\ht0\hbox to 0pt{\hss$\m@th#1\smile$\hss}%
 \kern\wd0}

\def\nindsym#1#2{%
 \setbox0=\hbox{$\m@th#1x$}%
 \kern\wd0%
 \hbox to 0pt{\hss$\m@th#1\not$\kern1.4\wd0\hss}
 \hbox to 0pt{\hss$\m@th#1\mid$\hbox to 0pt{$\m@th#1^{#2}$\hss}\hss}%
 \lower.9\ht0\hbox to 0pt{\hss$\m@th#1\smile$\hss}%
 \kern\wd0}

\title{On $\omega$-categorical structures with few finite substructures}
\author{Pierre Simon\footnote{Partially supported by NSF (grant no. 1665491) and a Sloan fellowship.}}

\date{\today}

\begin{document}
\maketitle

\begin{abstract}
We establish new results on the possible growth rates for the sequence $(f_n)$ counting the number of orbits of a given oligomorphic group on unordered sets of size $n$. Macpherson showed that for primitive actions, the growth is at least exponential (if the sequence is not constant equal to 1). The best lower bound previously known for the base of the exponential was obtained by Merola. We establishing the optimal value of 2 in the case where the structure is unstable. This allows us to improve on Merola's bound and also obtain the optimal value for structures homogeneous in a finite relational language. Finally, we show that the study of sequences $(f_n)$ of sub-exponential growth reduces to the $\omega$-stable case.
\end{abstract}

\section{Introduction}

A permutation group $G$ acting on an infinite set $\Omega$ is said to be \emph{oligomorphic} if the number of orbits on $\Omega^n$ is finite for all integers $n$. Given such an action, let $f_n$ denote the number of orbits of $G$ on (unordered) subsets of $\Omega$ of size $n$. The behavior of the sequence $(f_n)$ has been studied in a number of papers, most notably by Cameron and Macpherson. For instance, Cameron has shown that it is non-decreasing and can have arbitrary fast growth (see \cite[Chapter 3]{Cameron:oligomorphic}; the first statement is also proved by Pouzet \cite{pouzet}). Permutation groups for which $f_n=1$ for all $n$ were classified by Cameron \cite{Cameron:linear}. There are exactly 5 of them (here and throughout, we take $\Omega$ to be countable): the full symmetric group, the group of order preserving (resp. order preserving and reversing) permutations of $(\mathbb Q,\leq)$, the group of order preserving (resp. order preserving and reversing) permutations of a dense circular order.  The corresponding structures are precisely the reducts of dense linear order (DLO). After that, there is a gap in the possible growth rates.

\begin{thm}[Macpherson \cite{Mac:orbits}]\label{th:macpherson}
	There is a constant $c>1$ such that if $G$ is a primitive oligomorphic group, then either $f_n$ is constant equal to 1, or
	\[f_n \geq \frac {c^n}{p(n)},\] for some polynomial $p$.
\end{thm}

Macpherson obtains $c=2^{1/5}\approx 1.148$. This was improved by Merola \cite{merola} to $c\approx 1.324$. It is known that $c$ cannot be greater than 2: the example is a so-called \emph{local order}, see \cite{Mac_survey}. This structure---and its reduct obtained by adding order-reversion bijections to the automorphism group---are very likely the only primitive structures that realize $c=2$. We expect the next possible value for $c$ to be approximately $2.483$, which is realized by the $C$- or $D$-structure associated to a binary tree.

\begin{conjecture}[Macpherson]
	One can take $c=2$ in Theorem \ref{th:macpherson}.
\end{conjecture}

Many more results and questions about the sequence $f_n$ can be found in \cite[Chapter 3]{Cameron:oligomorphic} and in the survey \cite[Section 6.3]{Mac_survey}.

In this paper, we approach this conjecture using model-theoretic tools: that is we see $G$ as the automorphism group of a first order structure on $\Omega$ (which we rename $M$). The fact that $G$ is oligomorphic translates to $M$ being $\omega$-categorical. Our main theorem is the following (the terminology is defined in the next section).

\begin{thm}\label{th:main}
	Let $M$ be an $\omega$-categorical primitive structure, then one of the following holds:
	\begin{enumerate}
		\item $M$ is bi-definable with one of the 5 reducts of DLO;
		\item $M$ is strictly stable (stable, not $\omega$-stable);
		\item we have \[f_n \geq \frac {2^n}{p(n)},\] for some polynomial $p$.
	\end{enumerate}	
\end{thm}

Merola's work \cite{merola} follows a division in cases as in Macpherson's original proof \cite{Mac:orbits}. The case for which the value $c$ is minimal is the 2-homogeneous, not 2-transitive case, which implies the existence of a definable tournament. Since there are no stable tournament, the previous theorem rules out this case and we immediately get an improvement on the value of $c$, to the second best value in Merola's argument.

\begin{cor}\label{cor:main better c}
	One can take $c\approx 1.576$ in Theorem \ref{th:macpherson}.
\end{cor}

The exact value that appears in the corollary is $t^{1/2}$, where $t\approx 2.483....$ is equal to $\lim (t_n)^{1/n}$, where $t_n$ is the number of binary trees with $n$ leaves. See \cite{Mac:orbits}.

A class of $\omega$-categorical structures of special importance is that of structures homogeneous in a finite relational language. Such structures cannot be strictly stable. In this case, we therefore confirm Macpherson's conjecture.

\begin{cor}
	Let $M$ be homogeneous in a finite relational language and primitive, then either $M$ is bi-definable with a reduct of DLO (and thus $f_n=1$), or \[f_n \geq \frac {2^n}{p(n)},\] for some polynomial $p$.	
\end{cor}

To prove the conjecture in full, it remains to deal with the strictly stable case. We know by work of Lachlan \cite{lachlan:stable-omegacat} that we then have a definable pseudoplane. It seems plausible that one could use this to obtain the bound $2^n/p(n)$ in this case also using a different set of arguments.

Our analysis also yields the following result.

\begin{thm}\label{th:main 2}
	Let $M$ be $\omega$-categorical and assume that for no polynomial $p(x)$ do we have $f_n \geq \phi^n/p(n)$, where $\phi\approx 1.618$ is the golden ratio, then there is a reduct $M_*$ of $M$ which is stable and such that $f_n(M_*)=f_n(M)$ for all $n$.
\end{thm}

The value $\phi$ is optimal here as witnessed by the structure consisting of an equivalence relation $E$ with classes of size 2 and a linear order on the quotient. For such an $M$, $(f_n)$ is the Fibonacci sequence, whereas any stable reduct has $f_n$ linear.

If we replace $\phi$ by the constant $c\approx 1.57$ that appears in Corollary \ref{cor:main better c}, then we can strengthen the conclusion to $M_*$ being $\omega$-stable. Thus the study of the possible sub-exponential growth rates for $f_n$ reduces to the $\omega$-stable case, for which strong structure theorems are known (see \cite{Hr_totallycategorical}). This seems like a promising approach to solving remaining questions on the sub-exponential regime: see \cite{Cameron:oligomorphic}.

We note that the case where $f_n$ has polynomial growth is studied in depth in forthcoming work by Falque and Thi\'ery, announced in \cite{FalThi}. The previous result in this case seems to appear, at least implicitly, in their work.

\smallskip
This paper is self-contained, except for the use of some classical results in model theory and the theorems of Macpherson and Merola mentioned above. We only assume that the reader is familiar with  the language of model theory.

\section{Preliminaries}

Throughout, $M$ is a countably infinite structure in a \emph{relational} language $L$, which could be finite or infinite. By default, the term definable will be used to mean definable without parameters, otherwise we say parameter-definable, or definable over some set $A$.

We will distinguish between elements of $M$, written $a,b,\ldots$, and tuples of such elements, denoted $\bar a,\bar b,\ldots$. If $D$ is a parameter-definable set, and $\bar c$ a tuple, we let $D(\bar c)$ denote the set of elements of $\bar c$ that lie in $D$.

\medskip

A structure $M$ is $\omega$-categorical if any two countable structures elementarily equivalent to $M$ are isomorphic. By the classical theorem of Ryll-Nardzewski, $M$ is $\omega$-categorical if and only if it has finitely many $n$ types, for all $n<\omega$. In group theory terms, this is equivalent to the automorphism group of $M$ being oligomorphic. Conversely, given an oligomorphic action of a group $G$ on a set $M$, we can make $M$ into a first order structure by adding a predicate to name every orbit of $G$ on $M^k$, for each $k$. The resulting structure admits $G$ as automorphism group and is $\omega$-categorical. See for instance \cite[Section 3.1]{Mac_survey} and reference therein.

If $M$ is $\omega$-categorical, then the automorphism group $G$, along with its action on $M$, determines the structure $M$ up to \emph{bi-definability}: two structures on a set $M$ are bi-definable if they have the same $\emptyset$-definable sets in all dimensions. In what follows, when we say that a structure $M$ \emph{is}, for instance, a dense linear order, we mean that it is bi-definable with a dense linear order.

For the rest of this paper, $M$ is assumed to be $\omega$-categorical.

\medskip

If $E$ is a definable equivalence relation on $M^k$ (or more generally, any definable subset of $M^k$) and $V=M^k/E$, then the projection map $\pi:M^k\to V$ is called an interpretable map. In this case, $V$ is called an imaginary set. It is equipped with a definable structure as follows: a definable subset of $V^n$ (or more generally $M^m \times V^n$) is a subset whose pullback to $M^{kn}$ (resp. $M^{m+kn}$) is definable. We let $M^{eq}$ be the multi-sorted structure having one sort for every imaginary set. Any automorphism of $M$ extends uniquely to an automorphism of $M^{eq}$.

If $D\subseteq M$ is parameter-definable, by a formula $\phi(x;\bar a)$, then we can define an equivalence relation $E_{\phi}(\bar x,\bar y)$ on tuples of size $|\bar a|$ by \[E_{\phi}(\bar b,\bar b') \iff (\forall x)(\phi(x;\bar b) \leftrightarrow \phi(x;\bar b')).\]

Let $e$ be the $E_\phi$-class of $\bar a$, seen as an element of $M^{eq}$. Then $e$ is a \emph{canonical parameter} of $D$: an automorphism of $M$ fixes $D$ setwise if and only if its unique extension to $M^{eq}$ fixes $e$.

\medskip

We say that an $A$-definable set $X$ is \emph{transitive} over $A$ if any two elements of $X$ have the same type over $A$. When $A$ is finite, this is equivalent to $\aut(X/A)$ acting transitively on $X$. Omitting $A$ means $A=\emptyset$.

We say that $M$ is primitive if $\aut(M)$ acts primitively on it, equivalently if $M$ has no $\emptyset$-definable equivalence relation other than equality and the trivial relation.

If $A\subset M$ is a finite set and $b\in M$ an element, we say that $b$ is algebraic (resp. definable) over $A$ if $b$ has finite orbit (resp. is fixed) under $\aut(M/A)$: the pointwise stabilizer of $A$. The set of elements algebraic (resp. definable) over $A$ is denoted $\acl(A)$ (resp. $\dcl(A)$). If $\bar b$ is a finite tuple, none of whose coordinates is algebraic over $A$, then we can find an infinite sequence $(\bar b_i:i<\omega)$ of pairwise disjoint tuples in the $\aut(M/A)$-orbit of $\bar b$. We say that two tuples $\bar a$ and $\bar b$ are equialgebraic if $\acl(\bar a)=\acl(\bar b)$, or equivalently if we have both $\bar a\in \acl(\bar b)$ and $\bar b\in \acl(\bar a)$.

\begin{defi}
	We say that $M$ has \emph{few substructures} if $M$ is $\omega$-categorical and for no polynomial $p(x)$ do we have \[ f_n(M) \geq \frac{2^n}{p(n)},\] for all $n<\omega$.
\end{defi}

If $\bar a$ is a finite tuple of elements of $M$, let $M_{\bar a}$ denote the expansion of $M$ obtained by naming every $\bar a$-definable subset of $M^k$, for all $k$.

\begin{lemma}\label{lem: naming parameters}
	Assume that $M$ has few substructures and $\bar a\in M^l$ is a finite tuple. Then the expansion $M_{\bar a}$ also has few substructures.
\end{lemma}
\begin{proof}
	Assume that \[f_n(M_{\bar a}) \geq \frac {2^n}{p(n)},\] for some polynomial $p(x)$ and all $n<\omega$. For $n\geq l$, let $g_n$ denote the number of substructures of $M_{\bar a}$ of size $n$ that contain $\bar a$. Then \[g_n \geq \frac {f_{n-l}(M_{\bar a})}{n^l},\] since every substructure of size $n-l$ can be extended to a substructure of size $n$ containing $\bar a$ by adding $l$ elements to it and the resulting map from structures of size $n-l$ to structures of size $l$ has fibers of size at most $n^l$.
	
	On the other hand, we have \[f_n(M) \geq \frac{g_n}{n^l}.\] To see this, consider the map sending a substructure of $M_{\bar a}$ of size $n$ containing $\bar a$ to its $L$-reduct. The fibers of this map have size at most $n^l$, since the isomorphism type of an $L_{\bar a}$-structure $B$ is determined by its reduct to $L$ along with an embedding of $\bar a$ into $B$. Putting it all together, we have, for $n\geq l$, \[f_n(M) \geq \frac{2^{n-l}}{p(n-l)n^{2l}},\] which shows that $M$ does not have few substructures.
\end{proof}

\subsection{Linear orders and their reducts}

Any transitive $\omega$-categorical linear order is isomorphic to $(\mathbb Q,\leq)$. The reducts of this structure follow from Cameron's result on highly homogeneous permutations groups \cite{Cameron:linear}: there are five of them. Apart from the trivial reduct to pure equality, there are three unstable proper reducts:

\begin{itemize}
	\item the generic betweenness relation $(\mathbb Q;B(x,y,z))$, where \[B(x,y,z) \leftrightarrow (x\leq y \leq z)\vee (z\leq y \leq x);\]
	\item the generic circular order $(\mathbb Q;C(x,y,z))$, where \[C(x,y,z) \leftrightarrow (x\leq y\leq z)\vee (z\leq x\leq y) \vee (y\leq z\leq x);\]
	\item the generic separation relation $(\mathbb Q;S(x,y,z,t))$, where \begin{eqnarray*}S(x,y,z,t) \leftrightarrow (C(x,y,z)\wedge C(y,z,t)\wedge C(z,t,x)\wedge C(t,x,y)) \vee \\ (C(t,z,y)\wedge C(z,y,x)\wedge C(y,x,t) \wedge C(x,t,z)).\end{eqnarray*}
\end{itemize}

The automorphism group of the betweenness relation is generated by the automorphism of the linear order along with a bijection that reverses the order, for instance $x\mapsto -x$. Similarly, the automorphism group of the separation relation is generated from that of the circular order along with an order-reversing bijection.

We will say that two (linear or circular) orders on some set $D$ agree up to reversal if they are either equal, or reverse of each other.



\subsection{Stability and NIP}

A partitioned formula $\phi(\bar x;\bar y)$ has the order property (in a given structure $M$) if for every $N<\omega$, we can find tuples $(\bar a_i,\bar b_i:i<N)$ of elements of $M$ such that \[M\models \phi(\bar a_i;\bar b_j) \iff i\leq j.\]

A structure $M$ is called \emph{stable} if no formula has the order property. By a theorem of Shelah, this is equivalent to no formula of the form $\phi(x;\bar y)$---where $x$ is just one variable---having the order property.

\medskip
A partitioned formula $\phi(\bar x;\bar y)$ has the independence property (in $M$) if for every $N<\omega$, we can find tuples $(\bar a_i:i<N)$ and $(\bar b_J:J\in \mathfrak P(N))$ such that
\[M\models \phi(\bar a_i;\bar b_J) \iff i\in J.\]

A structure $M$ is called \emph{NIP}, or dependent, if no formula $\phi(\bar x;\bar y)$ has the independence property. Again, this is equivalent to no formula of the form $\phi(x;\bar y)$ having the independence property.

\begin{fact}[Macpherson \cite{Mac:rapid_growth}]\label{fact: few substructures in NIP}
	If $M$ is $\omega$-categorical and $f_n(M) = o(2^{p(n)})$ for all polynomials $p(x)$ of degree 2, then $M$ is NIP.
\end{fact}

\begin{corollary}\label{cor: NIP}
	If $M$ has few substructures, then $M$ is NIP.
\end{corollary}

One of the first result proved by Shelah about NIP structures is that the unstable ones interpret a quasi-order with infinite chains:

\begin{fact}[Shelah \cite{Sh10}]\label{fact: partial order}
	If $M$ is NIP, unstable, then there is a parameter-definable quasi-order $\leq$ on $M$ with an infinite chain.
\end{fact}

The fact that the order can be taken on $M$ itself is not stated in Shelah's paper, but follows from the proof and the fact mentioned above that there is an unstable formula of the form $\phi(x;\bar y)$. This is slightly more explicit in the presentation of \cite[Theorem 2.67]{NIPbook}.

\medskip
The starting point for the analysis in this paper is the following result (itself a special case of a more general result on general NIP unstable structures).

\begin{fact}[\cite{linear_orders}]\label{fact: producing a linear order}
	If $M$ is $\omega$-categorical, NIP and unstable, then there is an interpretable map $\pi:M\to V$, such that $V$ infinite and admits a definable linear order.
\end{fact}

In fact, we will see in Section \ref{sec: linear} how to deduce this result from Fact \ref{fact: partial order} in the special case where $M$ has few substructures, so that the present paper is not dependent on \cite{linear_orders}.

\subsubsection{$\omega$-stability}

The notion of $\omega$-stable structure will appear a couple of times in this paper, but will not play an important role and can be treated as a black box. For the sake of completeness, we give a definition and state the few facts that we will use.

\begin{defi}
	Let $M$ be $\omega$-categorical and $D\subseteq M^k$ a parameter-definable set. We define the Morley rank of $D$---denoted $MR(D)$---which takes values in $Ord\cup \{\infty\}$ by induction as follows:
	
	\begin{itemize}
		\item $MR(D)\geq 0$ if and only if $D$ is non-empty;
		\item $MR(D)\geq \alpha+1$, for an ordinal $\alpha$, if and only if there are disjoint parameter-definable subsets $F_0,F_1,F_2,\ldots$ of $D$ each of Morley rank at least $\alpha$;
		\item $MR(D)\geq \lambda$, for $\lambda$ a limit ordinal if and only if $MR(D)\geq \alpha$ for all $\alpha<\lambda$.
	\end{itemize}
	
	We say that $M$ is $\omega$-stable if $MR(M)<\alpha$ for some ordinal $\alpha$.
\end{defi}

If $M$ is $\omega$-stable, then it is stable. The $\omega$-stable, $\omega$-categorical structures have been studied in a sequence of fundamental works (see e.g., \cite{Zilber_book}, \cite{CHL}, \cite{Lachlan:coord}, \cite{Hr_totallycategorical}) and are now very well understood. Making use of those results, Macpherson \cite{Mac:rapid_growth}, showed the following (which follows from Theorems 1.4 and 1.5 there).

\begin{fact}\label{fact: omega stable case}
	If $M$ is $\omega$-categorical, $\omega$-stable and primitive, then either $M$ is an infinite set with no structure, or $f_n \geq \lfloor n/5 \rfloor !$, for $n$ large enough.
\end{fact}

\begin{corollary}\label{cor: small growth stable}
	Assume that for some polynomial $p(x)$, we have $f_n(M)\leq c^n/p(n)$, where $c$ is such that Theorem \ref{th:macpherson} applies, then $M$ is either $\omega$-stable, or unstable.
\end{corollary}
\begin{proof}
	We prove the result by induction on the number of 2-types of $M$. Assume that $M$ is stable. If $M$ is primitive, then $M$ is a pure infinite set by Theorem \ref{th:macpherson}. This is the case in particular if $M$ has a unique 2-type. Assume that $M$ is not primitive, then there is an interpretable $\pi:M\to N$ where $N$ is infinite and primitive. We have $f_n(N)\leq f_n(M)$ for each $n$. Stability is preserved under interpretation, so $N$ is a pure infinite set. Each fiber of $\pi$, equipped with the induced structure, is a stable structure with fewer 2-types and fewer substructures than $M$. Therefore by induction, each fiber is $\omega$-stable. It follows that $M$ itself is $\omega$-stable.
\end{proof}

An $\omega$-categorical structure that is stable and not $\omega$-stable, will be called \emph{strictly stable}. Usually this term is used to denote stable non-superstable structures, but superstability and $\omega$-stability are equivalent for $\omega$-categorical structures (\cite{lachlan:stable-omegacat}). Examples of strictly stable $\omega$-categorical structures can be obtained using a so-called Hrushovski construction (see e.g. Wagner's paper in \cite{automorphisms}). Essentially the only general result about them is the theorem of Lachlan \cite{lachlan:stable-omegacat} stating that such a structure interprets a pseudoplane. In fact, it seems likely that one could use this result to prove that those structure cannot have few substructures, but we do not pursue this here.

\section{Tameness of definable orders}

In this section, we assume that $M$ has few substructures.

\begin{lemma}\label{lem: finite fibers}
	Let $\pi:D\to V$ be an interpretable function, with $D\subseteq M$ and assume that $V$ is infinite, transitive and admits a definable linear order. Then $\pi$ has finite fibers.
\end{lemma}
\begin{proof}
	Fix $n<\omega$ and let $\sigma=(a_0,\ldots,a_{k-1})$ be a finite sequence of positive integers with $n=a_0+\cdots +a_{k-1}$. We associate to $\sigma$ a finite substructure $A_\sigma\subseteq D$ of size $n$ such that $\pi(A_\sigma)=\{b_0,\ldots ,b_{k-1}\}$ has size $k$, $b_0<\cdots<b_{k-1}$, and $\pi^{-1}(b_i)$ has size $a_i$ for each $i<k$. Then for $\sigma \neq \sigma'$, the substructures $A_{\sigma}$ and $A_{\sigma'}$ are not isomorphic. Since they are $2^{n-1}$ possibilities for $\sigma$, this contradicts $M$ having few substructures.
\end{proof}

\begin{lemme}\label{lem: weak o-minimal}
	Let $\pi:D\to V$ be an interpretable function, with $D\subseteq M$ and assume that $V$ admits a definable linear order. Then any parameter-definable subset of $V$ is a finite union of convex subsets.
\end{lemme}
\begin{proof}
	Let $X\subseteq V$ be parameter-definable and assume that $X$ is not a finite union of convex sets. Using Lemma \ref{lem: naming parameters}, we can assume that $X$ is definable over $\emptyset$. For any $\sigma \in {}^n 2$, we can find points $(a^\sigma_i:i<n)$ in $D$ such that $\pi(a^{\sigma}_i)<\pi(a^{\sigma}_{i+1})$ and $\pi(a^\sigma_i)$ is in $X$ if and only if $\sigma(i)=0$. The sets $A_\sigma = \{a^\sigma_i:i<n\}$ give $2^n$ distinct substructures of size $n$, which contradicts the hypothesis on $M$.
\end{proof}

If $V$ is equipped with a definable linear order $\leq$, we let $\overline V$ denote the completion of $V$, that is the set of initial segments $W\subseteq V$ which are non-empty and not equal to $V$, ordered by inclusion. The ordered set $V$ embeds naturally in $\overline V$ by $a\mapsto \{x:x\leq a\}$. If $D$ is any definable set, a function $f:D\to \overline V$ is said to be definable if the set $\{(x,y)\in D\times V : y\leq f(x)\}$ is a definable subset of $D \times V$.

\begin{lemma}\label{lem: no functions to orders}
	Let $X\subseteq M$ be definable and transitive, $D\subseteq M$, and $\pi:D\to V$ interpretable, where $V$ is infinite, transitive and admits a definable linear order $\leq$. Let $f:X\to \overline V$ be a definable function. Then $X= D$ and $f = \pi$.
\end{lemma}
\begin{proof}
	Assume that $D\setminus X$ is not empty. Then, by transitivity of $V$, $\pi$ maps $D\setminus X$ onto $V$. Fix $n<\omega$. For $\sigma \in {}^n 2$, let $(a^\sigma_i:i<n)$ be a sequence of elements of $M$ such that:
	
	$\bullet$ $a^\sigma_i \in D\setminus X$ if $\sigma(i)=0$, and $a^\sigma_i \in X$ if $\sigma(i)=1$;
	
	$\bullet$ $g_i(a^\sigma_i)<g_{i+1}(a^\sigma_{i+1})$, for each $i<n-1$, where $g_\bullet$ is equal to $f$ if applied to an element of $X$ and to $\pi$ otherwise.
	
	Set $A_\sigma = \{a^\sigma_i:i<n\}$. Then the substructures $A_\sigma$ are pairwise non-isomorphic, and hence contradict $M$ having few substructures. Therefore $X= D$.
	
	
	Assume that $f\neq \pi$. Then without loss of generality (using transitivity of $X$), we have $f(a)>\pi(a)$ for all $a\in X$. Let $\sigma \in {}^n 2$ and build a sequence $(c^{\sigma}_i:i<n)$ of elements of $X$ inductively: let $c{\sigma}_0$ be any element of $X$. Having built $c{\sigma}_i$, choose $c{\sigma}_{i+1}\in X$ so that $\pi(c{\sigma}_{i+1})>\pi(c{\sigma}_i)$ and \[\pi(c{\sigma}_{i+1})<f(c{\sigma}_i) \iff \sigma(i)=0.\]

This again gives us $2^n$ many substructures of size $n$.
\end{proof}

It follows in particular that if $V$ is as above, then any parameter-definable subset of $V$ has a maximal element, hence using Lemma \ref{lem: weak o-minimal}, the induced structure on $V$ is o-minimal: any parameter-definable subset of $V$ is a finite union of intervals.

\ignore{
\begin{prop}\label{prop: lines are dlo}
	Let $(V,\leq)$ be a linear order interpreted in dimension 1 over some $A$. Then there is a decomposition $V=V_1+\cdots +V_n$ into convex pieces such that over $A$, each $V_i$ is either finite or a pure DLO and there are no extra $A$-definable relations on $V$.
\end{prop}
\begin{proof}
	By Lemma \ref{lem: weak o-minimal}, any definable subset of $V$ is a finite union of convex sets. Let $V=V_1+\cdots+V_n$ be the decomposition of $V$ into complete types over $A$. We have to show that there is no extra structure on $V$ definable over $A$. We first show that there are no extra two types. Let $\phi(x,y)$ be an $A$-definable subset of $V^2$. For a fixed $a\in V$, the set $\phi(a,y)$ is a definable subset of $V$ and as such is a finite union of convex subsets. Any cut defined by those subsets is an element of $\overline V$ definable over $Aa$. By Lemma \ref{lem: no functions to orders}, such an element must be either equal to $a$ or definable over $A$ (indeed the projection to $V$ of an element of $A$). It follows that the intersection of $\phi(a,y)$ with any $V_i$ is equal to the intersection with $V_i$ of a set of the form $a \square y$, where $\square \in \{=,\neq,<,>,\leq,\geq\}$. Therefore $\phi(x,y)$ is quantifier-free definable from unary predicates for the $V_i$'s and the order.
	
	Since is true over any $A$, so by an immediate induction, we obtain that any subset of $V^n$ definable over $A$ is quantifier-free definable from unary predicates for the $V_i$'s and the order, which proves the lemma.
\end{proof}
}

\begin{prop}\label{prop: lines are dlo}
	Let $D\subseteq M$ be a definable subset and $\pi:D \to V$ an interpretable map, where $V$ is infinite, transitive and admits a definable linear order. Then $\pi$ has finite fibers and any formula $\phi(x_1,\ldots,x_n)$ definable over some $\bar b\in M^k$ is equivalent to a finite boolean combination of formulas of the form:
	
	$\bullet$ $\pi(x_i) \mathrel{\square} \pi(x_j)$, where $\square \in \{=,\leq\}$;

	$\bullet$ $\pi(x_i)< \pi(a)$, for some $a\in D(\bar b)$;

	$\bullet$ $x_i = a$, for some $a\in \pi^{-1}(\pi(D(\bar b))$.
\end{prop}
\begin{proof}
	By Lemma \ref{lem: finite fibers}, $\pi$ has finite fibers.
	Let $\bar c=(c_0,\ldots,c_k)$ be a finite tuple of elements of $M$ and assume that some element $v\in V$ is definable (equivalently, algebraic) over $\bar c$. Take $l$ minimal so that $v\in \dcl(c_0,\ldots,c_l)$. Set $\bar c_* = (c_0,\ldots,c_{l-1})$. Then over $\bar c_*$, we have a definable function $f$ defined on the locus $X$ of $\tp(c_l/\bar c_*)$ with image in some infinite transitive $W\subseteq V$ and sending $c_l$ to $v$. By Lemma \ref{lem: no functions to orders} applied in $M_{\bar c_*}$ (using Lemma \ref{lem: naming parameters}), we conclude that $c_l\in D$ and $f=\pi|_X$. It follows that $\dcl(\bar c)\cap V=\pi(D(\bar c))$ and therefore also \[\acl(\bar c)\cap D = \pi^{-1}(\pi(D(\bar c))).\]
	
	Since the structure on $V$ is o-minimal, any subset of $V$ definable over some parameters $\bar c$ is a boolean combination of intervals with endpoints in $\pi(D(\bar c))$. Hence by induction on $n$, any subset $X\subseteq V^n$ definable over $\bar c$ is a boolean combination of definable sets of the form
	
	$\bullet$ $x_i \mathrel{\square} c_*$, for $\square \in \{=,\leq\}$ and $c_* \in \pi(D(\bar c))$;
	
	$\bullet$ $x_i \mathrel{\square} x_j$, for $\square \in \{=,\leq\}$.

	Finally, let $X\subseteq D$ be definable and transitive over some parameters $A$. Assume that $X$ is infinite and let $W=\pi(X)$ and $D'=\pi^{-1}(W)$. Then $W$ is transitive over $A$. By Lemma \ref{lem: no functions to orders} (adding $A$ to the base, using Lemma \ref{lem: naming parameters}), $X=D'$. It follows that any parameter-definable subset of $D$ is a union of a finite set and the pullback of a definable subset of $V$. By the previous discussion, any finite subset of $D$ definable over some $\bar c$ is in $\pi^{-1}(\pi(D(\bar c))$. The result now follows.
\end{proof}

\begin{prop}\label{prop: circles are dlo}
	Let $D\subseteq M$ be a definable subset and $\pi:D \to V$ an interpretable map, where $V$ is infinite, transitive and admits a definable circular order $C(x,y,z)$. Then $\pi$ has finite fibers and any formula $\phi(x_1,\ldots,x_n)$ over some $\bar b\in M^k$ is equivalent to a finite boolean combination of formulas of the form:
	
	$\bullet$ $\pi(x_i) = \pi(x_j)$;
	
	$\bullet$ $C(\pi(x_i),\pi(x_j),\pi(x_k))$;
	
	$\bullet$ $C(\pi(a_0),\pi(x_i),\pi(x_j))$, for some $a_0\in D(\bar b)$;

	$\bullet$ $C(\pi(a_0),\pi(a_1),\pi(x_i))$, for some $a_0,a_1\in D(\bar b)$;	

	$\bullet$ $x_i = a$, for some $a\in \pi^{-1}(\pi(D(\bar b))$.
\end{prop}
\begin{proof}
	Let $\bar b$ be any finite tuple of parameters from $M$ and let $X\subseteq V$ be definable and transitive over $\bar b$. If $a\in D$, then $V$ admits an $a$-definable linear order. Therefore by Lemma \ref{lem: weak o-minimal}, and using Lemma \ref{lem: naming parameters} to work over $a$, $X$ is a finite union of intervals of $V$. Let $\bar d$ enumerate the endpoints of those intervals. Then $\bar d$ is algebraic over $\bar b$ and hence definable over $\bar b\hat{~}a$. By transitivity of $V$, we can choose $a$ so that no element of $\bar d$ is algebraic over $a$. But then each such element lies in some infinite linearly-ordered set, definable and transitive over $a$. By Proposition \ref{prop: lines are dlo}, $\bar d \in \pi(D(\bar b))$.
	
	The conclusion now follows by induction as in the linearly ordered case.
\end{proof}

\begin{cor}\label{cor: no structure on orders}
	Let $D\subseteq M$ be definable and $\pi:D \to V$ an interpretable map, where $V$ is infinite, transitive and admits a definable separation relation. Let $\bar b$ be a tuple of elements of $M$, disjoint from $D$. Then $V$ is transitive over $\bar b$ and its structure over $\bar b$ is precisely one of the 4 unstable reducts of DLO.
\end{cor}
\begin{proof}
	If $V$ admits a definable linear or circular order, then the result follows at once from the two previous propositions. Assume that $V$ admits a definable betweenness relation. Then there are two linear orders that induce it, and each are definable over a finite set of parameters. Hence by Lemma \ref{lem: naming parameters}, if we add a predicate for either of those two orders, the structure still has few substructures and furthermore $V$ is still transitive. Then the structure on $V$ over $\bar b$ in the expanded language is a pure DLO. In the reduct, it is therefore a reduct of DLO. The case where $V$ has a definable separation relation is similar.
\end{proof}

\section{Glueing orders}

We again assume that $M$ has few substructures.

A definable set $D_{\bar a}\subseteq M$, definable over some $\bar a$ is \emph{almost linear} (over $\bar a$) if there is a $\bar a$-definable equivalence relation $E_{\bar a}$ with infinitely many classes and a $\bar a$-definable linear order on the quotient $V_{\bar a}:=D_{\bar a}/E_{\bar a}$. By Lemma \ref{lem: finite fibers}, if $V_{\bar a}$ is transitive over $\bar a$ (in particular, if $D_{\bar a}$ is transitive over $\bar a$), then the $E_{\bar a}$-classes are finite. By an interval of $D_{\bar a}$, we mean the pullback of an interval of $V_{\bar a}$ to $D_{\bar a}$. We will sometimes think of $\leq_{\bar a}$ as defining a quasi-order on $D_{\bar a}$.

\begin{lemme}\label{lem: equialgebraic fibers}
	Let $D_{\bar a}\subseteq M$ be almost linear and transitive over some tuple $\bar a$ witnessed by $E_{\bar a}$ and $(V_{\bar a},\leq_{\bar a})$. Let $c \in D_{\bar a}$. Then for any $c'\in M$, the following are equivalent:
	
	\begin{enumerate}
		\item $c'\in D_{\bar a}$ and $E_{\bar a}$-equivalent to $c$;
		\item $c\in \acl(c')$.
	\end{enumerate}
	
	\noindent
	In particular the $E_{\bar a}$-class of $c$ is composed precisely of the elements equialgebraic with $c$.
\end{lemme}
\begin{proof}
	Let $c\in D_{\bar a}$. Enumerate the other elements in the $E_{\bar a}$-class of $c$ as $\bar c_0\hat{~}\bar c_1$, where $\bar c_0$ is composed of points in $\acl(c)$ and $\bar c_1$ of points outside of it. There is an automorphism $\sigma$ of $M$ fixing $c$ such that no element in $\sigma(\bar c_1)$ is algebraic over $\bar a\hat{~} c$ (since we can find infinitely many disjoint conjugates of that tuple). Let $\bar a'=\sigma(\bar a)$.
		
	The $E_{\bar a'}$-class of $c$ and the $E_{\bar a}$-class of $c$ intersect in precisely $c\hat{~}\bar c_0$. It follows that $c$ is algebraic over $\bar a\hat{~}\bar a'$. By Proposition \ref{prop: lines are dlo}, $\bar a$ must contain a point in the $E_{\bar a'}$-class of $c$. This point is algebraic over $\bar a\hat{~}c$, hence is in $\bar c_0$. This contradicts transitivity of $D_{\bar a}$.
	
	We have shown that (1) implies that $c'\in \acl(c)$. By symmetry of the roles of $c$ and $c'$, we also have $c\in \acl(c')$. Conversely, if $c$ is algebraic over some $c'\in M$, then by Proposition \ref{prop: lines are dlo}, $c'$ is in the $E_{\bar a}$-class of $c$.
\end{proof}

\begin{lemme}\label{lem: uniqueness of germs}
	Let $D_{\bar a}$ be as in the previous lemma and let also $D'_{\bar b}$ be transitive and almost linear over $\bar b$, witnessed by $E'_{\bar b}$ and $(V'_{\bar b},\leq_{\bar b})$. Assume that there is $c\in D_{\bar a} \cap D'_{\bar b}$, then $D_{\bar a}$ and $D_{\bar b}$ agree up to reversal on an open interval of $c$ (in the sense of both $\leq_{\bar a}$ and $\leq'_{\bar b}$).
\end{lemme}
\begin{proof}
	The $E_{\bar a}$ and the $E'_{\bar b}$-classes of $c$  coincide since they are both equal to the set of elements equialgebraic to $c$. Applying Lemma \ref{lem: equialgebraic fibers} in the expansion $M_{\bar a}$, we see that if $c$ is algebraic (equiv. definable) over $\bar a\hat{~}\bar b$, then $\bar b$ contains some element in the $E_{\bar a}$-class of $c$. This implies that $c$ is algebraic over $\bar b$, which contradicts transitivity of $D'_{\bar b}$. Hence $c$ is not algebraic over $\bar a\hat{~}\bar b$. Therefore $c$ is not an end point of the intersection $D_{\bar a}\cap D'_{\bar b}$ in $D_{\bar a}$, and that intersection contains an open interval $I$ in the sense of $D_{\bar a}$. We may assume that $I$ is transitive over $\bar a\hat{~}\bar b$. By Proposition \ref{prop: lines are dlo}, the quasi-order induced by $D'_{\bar b}$ on $I$ is either the same or the reverse of that induced by $D_{\bar a}$. By exchanging the roles of $D_{\bar a}$ and $D'_{\bar b}$, we see that $I$ is also an interval of $D'_{\bar b}$.
\end{proof}

Let $\mathcal L$ denote the set of parameter-definable sets $D$ such that $D$ is almost linear and transitive over some finite tuple.

\begin{lemma}\label{lem: no forking}
	Let $D_1, D_2\in \mathcal L$ and assume that $D_1$ and $D_2$ have some interval $I$ in common, with the same induced (quasi-)order. Take $I$ to be maximal such. Then $I$ is cofinal either in $D_1$ or in $D_2$.
\end{lemma}
\begin{proof}
	Let $D_1$, $D_2$ be transitive almost linear over parameters $\bar a_1$ and $\bar a_2$ respectively. Assume that $I$ is cofinal neither in $D_1$ nor in $D_2$ and let $c\in D_1$ (resp. $d\in D_2$) be a minimal upper bound of $I$. Then $c$ is definable over $\bar a_1\hat{~}\bar a_2$. Applying Lemma \ref{lem: no functions to orders} to the expansion $M_{\bar a_1}$, we see that $c$ is interalgebraic (over $\emptyset$) with some element of $\bar a_2$.
	
	Let $d_0=d > d_1 > d_2 > \cdots$ be a decreasing sequence of elements of $D_2$. By transitivity of $D_2$ over $\bar a_2$, for each $i$, there is an automorphism $\sigma_i\in \aut(M/\bar a_2)$ sending $d$ to $d_i$. Let $D_{1,i}$ be the image of $D_1$ under $\sigma_i$ and let $I_i$ be the image of $I$ under $\sigma_i$. Then $c_i:=\sigma_i(c)$ is a minimal upper bound of $I_i$ in $D_{1,i}$. We claim that the elements $c_i$ are pairwise distinct. To see this, assume say that $c_1=c$. Then by Lemma \ref{lem: uniqueness of germs}, $D_1 \cap D_{1,1}$ contains an open neighborhood of $c=c_1$ in the sense of both $D_1$ and $D_{1,1}$. But this is absurd since small enough neighborhoods of $c$ in $D_1$ and $c_1$ in $D_{1,1}$ intersect $D_2$ in two disjoint, non-empty, intervals.

	The elements $c_i$ being conjugates of $c$ over $\bar a_2$ are all algebraic over $\bar a_2$, but there are only finitely many points algebraic over $\bar a_2$, hence we have reached a contradiction.
\end{proof}

\begin{cor}\label{cor: no forking}
	Let $D_0$, $D_1\in \mathcal L$ and let $I=D_1 \cap D_2$. Then one of the following holds:
	
	\begin{enumerate}
		\item $I=\emptyset$;
		\item for some $i<2$, $I$ is an initial segment of $D_i$ and an final segment of $D_{1-i}$ and the two orders agree on $I$;
		\item $I$ is an initial (resp. final) segment of both $D_0$ and $D_{1}$ and the two orders are opposite of each other on $I$;
		\item for some $i<2$, $I=I_1 \cup I_2$, where $I_1$ is an initial segment of $D_i$ and an final segment of $D_{1-i}$ and $I_2$ is a final segment of $D_i$ and an initial segment of $D_{1-i}$, and the two orders agree on each of $I_1$ and $I_2$;
		\item $I=I_1 \cup I_2$, where $I_1$ is an initial segment of both $D_0$ and $D_1$ and $I_2$ a final segment of $D_0$ and $D_1$ and the two orders are opposite of each other on each of $I_1$ and $I_2$.
	\end{enumerate}
\end{cor}
\begin{proof}
	Assume that $I$ is non-empty. By Lemma \ref {lem: uniqueness of germs}, there is $J\subseteq I$ which is an open interval of both $D_0$ and $D_1$, and the two orders on $J$ on equal or opposite of each other. Take $J$ maximal such. Replacing $D_1$ by the same set with the opposite order, we can assume that the two orders agree on $J$.  By Lemma \ref{lem: no forking}, for some $i<2$, $J$ is an initial segment of $D_i$ and a final segment of $D_{1-i}$. If $I=J$, we are done, otherwise let $J'$ be a maximal interval in $I\setminus J$. Then by Lemma \ref{lem: no forking} again, $J'$ must be an initial segment of $D_{1-i}$ and  a final segment of $D_i$. There is no possibility left for another maximal interval of $I$, hence $I=J \cup J'$. This finishes the proof.
\end{proof}



If $D_0,D_1 \in \mathcal L$, write $D_0 \unlhd D_1$ if a final segment of $D_0$ coincides with an initial segment of $D_1$, with the same induced orders. Say that a point $c\in M$ is of order type if there is some $D \in \mathcal L$ containing $c$. By Lemma \ref{lem: uniqueness of germs}, if $c$ is of order type, then any two $D,D'\in \mathcal L$ containing $c$ coincide, up to reversal, on a neighborhood of $c$. Let $\Omega \subseteq M$ be the set of points of order type. Define an equivalence relation $E_0$ on $\Omega$ as follows: $a$ and $b$ are $E_0$-equivalent if there is a sequence $D_1,\ldots,D_n$ of elements of $\mathcal L$ such that $D_i \unlhd D_{i+1}$, $a\in D_1$ and $b\in D_n$.

Let $\mathcal E$ be an $E_0$-equivalence class. Call $\mathcal E$ circular if there is a finite sequence $D_0 \unlhd D_1 \unlhd \cdots \unlhd D_n \unlhd D_0$, otherwise call $\mathcal E$ linear. Assume that $\mathcal E$ is circular as witnessed by $D_0,\ldots, D_n$. The sequence $(D_0,\ldots,D_n, D_0)$ induces a circular order $C$ on $\mathcal E$ which coincides locally with each order $D_i$. If $D\in \mathcal L$ intersects $\mathcal E$, then by Corollary \ref{cor: no forking}, $D$ is an interval of $(\mathcal E,C)$  and the order on $D$ coincides up to reversal with the one induced by $C$. It follows that any other sequence $D'_0,\ldots,D'_{n'}$ that witnesses circularity of $\mathcal E$ defines on $\mathcal E$ either the circular order $C$, or its reverse. Hence there is on $\mathcal E$ a separation relation $S$ invariant over any automorphism that preserves $\mathcal E$ setwise. In other words, $S$ is definable over a canonical parameter for the class $\mathcal E$.

If now the class $\mathcal E$ is linear, then a similar argument shows that there is a betweenness relation $B$ on $\mathcal E$ definable over a canonical base for $\mathcal E$ such that for any $D\in \mathcal L$ included in $\mathcal E$, the restriction of $B$ to $D$ is compatible with the order on $D$.

Let $e$ be a canonical parameter for $\mathcal E$. Then either a separation relation or a betweenness relation on $\mathcal E$ is definable over $e$. If $D\in \mathcal L$ is definable and transitive over $\bar a$, then $e\in \dcl^{eq}(\bar a)$. In particular, any two elements of $D$ have the same type over $e$. By definition of $E_0$, it follows that any $a,b\in \mathcal E$ have the same type over $e$, so $\mathcal E$ is transitive over $e$. If $A\subseteq M$ is any set of parameters disjoint from the class $\mathcal E$ and let $e$ be an imaginary code for $\mathcal E$, then by Corollary \ref{cor: no structure on orders} the structure induced by $Ae$ on $\mathcal E$ is precisely one of the four unstable reducts of DLO.

\section{Conclusions}

\begin{thm}\label{th: main decomposition}
	Assume that $M$ has few substructures. Then there is a $\emptyset$-definable equivalence relation $F$ with finite classes on $M$, a $\emptyset$-definable equivalence relation $E$ on the quotient $M_0:=M/F$ such that the following holds: Let $N:=M_0/E$ and let $M_*$ denote the reduct of $M_0$ to pullbacks of definable sets on $N$, then $M_*$ is stable and $f_n(M_*)=f_n(M_0)$ for all $n$.
\end{thm}
\begin{proof}
	Let $F$ be the equivalence relation of equialgebraicity on points of order type, extended by equality on the rest of the structure. Then $F$ has finite fibers. Let $M_0$ be the quotient $M/F$. Seeing $M_0$ as a structure in its own right (equipped with the induced structure from $M$), we have $f_n(M_0)\leq f_n(M)$ for each $n$: since $F$ is $\emptyset$-definable, the quotient by $F$ induces a well-defined map from finite substructures of $M$ to finite substructures of $M_0$, such that any substructure of $M_0$ of size $n$ lifts to at least one substructure of $M$ of size $n$. It follows that $M_0$ has few substructures and from now on, we work in $M_0$.
		
	 Let $E$ be the equivalence relation $E_0$ defined above on points of order type, extended by equality outside and set $N=M_0/E$. Then $N$ with the induced structure is stable: it is NIP, since $M$ is, and if it were unstable, then by either Fact \ref{fact: producing a linear order} or Theorem \ref{th: producing a linear order}, there would be an interpretable map $\pi:N \to V$ to a linear order, hence a point $a\in N$ order type in the sense of $N$. If the fiber of $a$ under $\pi$ is infinite, this contradicts Lemma \ref{lem: finite fibers}. Hence the fiber above $a$ has one element, which we again call $a$, but then $a$ is of order type in $M_0$, contradicting the construction of $E$.
	 
	 Define $M_*$ as in the statement of the theorem: that is $M_*$ is the reduct of $M$ obtained by naming all pullbacks of definable subsets of $N^k$ (so in particular $E$ is named as the pullback of equality on $N^2$). Let $L_*$ be the corresponding language. Note that the points in $N$ whose $E$-class is infinite is definable. Hence $M_*$ is interpretable in $N$: it is obtained from $N$ by blowing up each such point to an infinite set. Therefore $M_*$ is stable.
	 
	 It remains to show that $f_n(M_*)=f_n(M_0)$ for all $n$. For this, we show that a finite substructure of $M_*$ has a unique expansion to a finite substructure of $M_0$, up to isomorphism. To this end, let $A,B\subseteq M_0$ be two finite substructures which have isomorphic $L_*$-reducts and we have to show that $A$ and $B$ are isomorphic. Let $\sigma:A\to B$ be a bijection which induces an ismorphism on the $L_*$-reducts. Write $A$ as a disjoint union $A=A_* \cup A_0 \cup \cdots \cup A_{d-1}$, where $A_*$ is composed of the points in $A$ which are not of order-type (in the sense of $M_0$) and each $A_i$, $i<d$, is composed of all the points in $A$ in some infinite $E$-class $e_i\in N$. Let similarly $B=B_*\cup B_0 \cup \cdots \cup B_{d-1}$, where $B_i = \sigma(A_i)$. Since $\sigma$ is $L_*$-elementary, it induces an isomorphism $\sigma_N$ in $N$ between the projections of $A$ and $B$ to $N$.
	 
	 We build inductively an isomorphism $\tilde \sigma$ between $A$ and $B$ (along with their images in $N$). Start by setting $\tilde \sigma$ to be equal to $\sigma_N$ on $\pi(A)$. In particular, this determines $\tilde \sigma$ on $A_*$ and on $\bar e$. Assume that $\tilde \sigma$ has been defined on $A_*\bar e A_{<l}$. Let $E_l=\pi^{-1}(e_l)$. The definable set $E_l$ is transitive over  $A_*\bar e A_{<l}$ since no of $A_*\bar e A_{<l}$ is in $E_l$. Therefore its structure is one of the four unstable reducts of DLO and as such, any two substructures of $E_l$ of the same size are isomorphic over $A_*\bar e A_{<l}$. Since $|A_l|=|B_l|$, we can extend $\tilde \sigma$ to send $A_l$ to $B_l$. This finishes the construction and the proof of the theorem.
\end{proof}

\begin{thm}\label{th: main conclusion}
	Assume that $M$ has few substructures and is primitive. Then either $M$ is strictly stable (stable, but not $\omega$-stable), or $M$ is one of the five reducts of DLO.
\end{thm}
\begin{proof}
	If $M$ is primitive, then the equivalence relation $F$ above is equality and $E$ is either trivial or equality. If it is trivial, then $M$ is composed of one $E$-class and hence is one of the four unstable reducts of DLO (up to bi-definability). If $E$ is equality, then $M$ is stable. By Fact \ref{fact: omega stable case}, if $M$ is $\omega$-stable, then it is pure equality. Otherwise, $M$ is strictly stable.
\end{proof}

\begin{corollary}
	Assume that $M$ is finitely homogeneous, primitive and has few substructures, then $M$ is one of the five reducts of DLO.
\end{corollary}
\begin{proof}
	Since every finitely homogeneous stable structure is $\omega$-stable, this follows from the previous theorem.
\end{proof}

For an integer $n$, let $t_n$ denote the number of rooted binary trees with $n$ terminal nodes (leaves). It is known (see \cite[p.55]{comtet}, or \cite[p.77]{Cameron:oligomorphic}) that $\lim (t_n)^{1/n} \approx 2.483$. Let $c$ be the square root of this number, $c\approx 1.576$.

\begin{corollary}
	Assume that $M$ is primitive and $f_n(M)\leq c^n$, for $c\approx 1.576$ as defined above, then $M$ is one of the five reducts of DLO.
\end{corollary}
\begin{proof}
	This follows from Theorem \ref{th: main conclusion}, using the results of \cite{merola}, and in particular the discussion at the very end of the paper. It is explained there that Theorem \ref{th:main} holds with $c$ roughly equal to $1.324$. This value is obtained in the case where the automorphism group is 2-homogeneous, not 2-transitive. In this case, there is a $\emptyset$-definable tournament (this is presented in Section 4 there). This cannot happen in a stable structure, since a tournament has the order property. We can therefore eliminate this case from the analysis of \cite{merola}. The next smallest constant comes from the case where the group is not 2-homogeneous. The value for that case was obtained by Macpherson \cite{Mac:orbits} and is equal to $c$ as defined above.
\end{proof}

\begin{corollary}\label{cor: stable reduct}
	Assume that for no polynomial $p(x)$ do we have $f_n(M)\geq \phi^n/p(n)$, where $\phi\approx 1.618$ is the golden ratio, then $M$ admits a stable reduct $M_*$ with $f_n(M_*)=f_n(M)$ for all $n$.
\end{corollary}
\begin{proof}
	Assume that $M$ is such that $f_n=o(\phi^n/p(n))$ for any polynomial $p$. Define $F$ as in the proof of Theorem \ref{th: main decomposition} and assume that $F$ is not equality. Then over some parameters, there is a definable subset $D\subseteq M$ and an interpretable map $\pi:D\to V$ with fibers of size $\geq 2$ and a definable infinite linear order on $V$. Using Lemma \ref{lem: naming parameters}, and losing a polynomial factor, we can assume all this is defined over $\emptyset$. Let $n<\omega$ and $\sigma=(a_1,\ldots,a_l)$ be a sequence of elements of $\{1,2\}$ which sum to $n$. Associate to $\sigma$ a substructure $A_\sigma\subseteq D$ of size $n$ such that $\pi(A_\sigma)$ has size $l$, equal to $e_1<\cdots <e_l$, say, and $A_\sigma \cap \pi^{-1}(e_i)$ has size $a_i$ for all $i\leq l$. Then one can recover $\sigma$ from the isomorphism type of $A_\sigma$. An easy induction shows that the number of possibilities for $\sigma$, for a given $n$, is equal to $F_n$: the $n$-th Fibonacci number. Hence $f_n\geq F_n \sim \phi^n/\sqrt 5$ and the result follows.
\end{proof}

In particular, the study of possible sub-exponential growth rates of the function $f_n$ reduces to the $\omega$-stable case.

\begin{corollary}
	If $f_n(M)$ grows slower than any exponential, then there is an $\omega$-stable structure $M_*$ with $f_n(M_*)=f_n(M)$ for all $n$.
\end{corollary}
\begin{proof}
	Let $c$ be as defined above. If $f_n(M)=o(c^n/p(n))$, for any polynomial $p(x)$, then by Corollary \ref{cor: stable reduct}, $M$ has a stable reduct $M_*$ with $f_n(M_*)=f_n(M)$ for all $n$. By Corollary \ref{cor: small growth stable}, $M_*$ is $\omega$-stable.
\end{proof}

\section{Producing a linear order}\label{sec: linear}

\subsection{Orders with bounded antichains}

We say that a partial order $(P,\leq)$ has bounded antichains if there is an integer $n$ such that any antichain of $P$ has size at most $n$. In the following theorem, we do not assume that $M$ is $\omega$-categorical.

\begin{thm}\label{th: bounded antichains}
Let $M$ be any structure and $(P,\leq)$ be an infinite definable partial order with bounded antichains, then, over some parameters, there is a definable equivalence relation $E$ on $P$ with infinitely many classes and a definable linear order on the quotient $P/E$.
\end{thm}
\begin{proof}
First note that it is enough to find some infinite definable $D\subseteq P$ on which the conclusion holds, since we can then extend the equivalence relation $E$ to $P$ by making $P\setminus D$ into a unique class, smaller than all elements of $D/E$.

We prove the result by induction on the size of the largest antichain. If that size is 1, then $P$ is already linear. Assume that $P$ has no antichain of size $n+1$. 
Assume that for some $a$, the set $V(a)$ of elements incomparable with $a$ is infinite. Then the poset $(V(a),\leq)$ is definable and has antichains of size at most $n$. Therefore by induction it interprets a linear order. Now assume this is not the case: for each $a$, $V(a)$ is finite, of size bounded by $k$ say. We prove by induction on $k$ that $P$ interprets a linear order.

For $a,b\in P$, we write $a\impl b$ if $b\in V(a)$ and $b$ is a maximal element in $V(a)$. We also define $a \tleq b$ if either $a\leq b$ or $a \impl b$.

\medskip
\underline{Claim}: If $a \tleq b$ and $b\tleq c$, then $a\tleq c$.

\emph{Proof}: If $a\tleq c$ does not hold, then either $c< a$, or $c\in V(a)$ is not maximal.

Case 1: $c<a$. If $a\leq b$, then also $c<b$, contradicting $b\tleq c$. So it must be that $a\impl b$. If $b\leq c$, then $b<a$ which is impossible, therefore also $b\impl c$. Then both $c,a$ are in $V(b)$ which contradicts maximality of $c$.

Case 2: $c\in V(a)$ is not maximal. Let $d\in V(a)$ such that $d>c$. Then $d$ cannot be in $V(b)$ (if $b\leq c$, then $b\leq d$ and if $b\impl c$, then $c$ must be maximal in $V(b)$). So $d$ is comparable to $b$. If $d\leq b$, then $c\leq b$ which is impossible. So $b< d$. In particular it cannot be that $a\leq b$, so $a\impl b$. But then $b$ is not maximal in $V(a)$ as $d$ is above it.

In both cases, we reached a contradiction. This completes the proof of the claim.

\medskip
It follows that $\tleq$ defines a quasi-order. Let $\equiv$ be the equivalence relation defined by $a\equiv b$ if $a\tleq b$ and $b\tleq a$. Then $a\equiv b$ implies that $a$ and $b$ are $\leq$-incomparable. Therefore the equivalence classes are finite. Furthermore, if $b$ is $\tleq$-incomparable to $a$, then it is also $\leq$-incomparable to $a$, and by construction, there is at least one $b\in V(a)$ which is $\tleq$-comparable to $a$.

To summarize, the interpretable structure $(P/\equiv , \tleq)$ is an infinite poset in which for every $a$, at most $k-1$ elements are incomparable to $a$. By induction, we can interpret an infinite linear order in it.
\end{proof}

\subsection{Linear orders in structures with few substructures}

In this section, we prove the following theorem, without reference to \cite{linear_orders}.

\begin{thm}\label{th: producing a linear order}
	Let $M$ have few substructures. Then either $M$ is stable or, over some set of parameters $A$, there is an interpretable map $\pi: M\to V$ with $V$ infinite and a definable linear order on $V$.
\end{thm}

\begin{proof}
Let $M$ have few substructures. By Fact \ref{fact: few substructures in NIP}, $M$ is NIP. Assume that $M$ is unstable. By Fact \ref{fact: partial order}, there is a parameter-definable quasi-order $\leq$ on $M$ with an infinite chain. Let $P$ denote the (interpretable) quotient partial-order.

Replacing $P$ by a subset of it if necessary, we assume that $(P,\leq)$ is definable and transitive over some base $A$, in particular, $P$ has no maximal or minimal element. Using Lemma \ref{lem: naming parameters}, without loss, $A=\emptyset$. If $P$ has no infinite antichain, then by $\omega$-categoricity, it has bounded antichains and the result follows from Theorem \ref{th: bounded antichains}. Assume that $P$ has an infinite antichain $C_0\subseteq P$. Let $a_0 \in C_0$ and consider the set $P_1 = \{x\in P : x>a_0\}$. Then $P_1$ is infinite by the transitivity assumption. If $P_1$ has bounded antichains, we are done. Otherwise, find an infinite antichain $C_1\subseteq P_1$ and take $a_1 \in C_1$. Continue in this way producing antichains $C_i$, $i<\omega$ and points $a_i\in C_i$ such that every element of $C_j$, $j>i$, is greater than $a_i$.

Fix $n<\omega$ and let $\sigma=(m_0,\ldots,m_{k-1})$ be a finite sequence of positive integers with $n=m_0+\cdots +m_{k-1}$. We associate to $\sigma$ a finite substructure $A_\sigma\subseteq D$ of size $n$ containing exactly $m_i$ points from the antichain $C_i$, including $a_i$, and no other point. Then given $A_\sigma$ one we can recover the sequence of integers $m_i$: there are exactly $m_i$ points in $A_\sigma$ with a maximal chain of size $i$ below them in $A_\sigma$. We thus obtain $2^{n-1}$ substructures of size $n$ and this contradicts $M$ having few substructures.
\end{proof}

\bibliography{tout.bib}

\end{document}